\documentclass[10pt leqno]{amsart}
\usepackage{color,transparent}
\usepackage{verbatim}
\usepackage{array}
\usepackage{tabularx}
\usepackage{booktabs}

\usepackage[bookmarks=true,%
    colorlinks=true,%
    linkcolor=blue,%
    citecolor=blue,%
    filecolor=blue,%
    menucolor=blue,%
    urlcolor=blue]{hyperref}
\usepackage{stmaryrd}
\usepackage{multirow}
\usepackage{float}
\usepackage{array}
\usepackage[shortlabels]{enumitem}
\usepackage{pdfsync}
\usepackage{tikz}
\usepackage{tikz-cd}
\usepackage{float}
\usepackage[T1]{fontenc}
\usepackage[utf8]{inputenc}
\usepackage[textsize=small]{todonotes}
\usepackage{svg}

\usepackage{graphicx}
\graphicspath{{mpdraws/}{draws/}}

\usepackage{latexsym,enumitem}
\usepackage{amssymb}

 \usepackage[sort,numbers]{natbib} 
 \usepackage{mathtools}
\usepackage[normalem]{ulem}
\usepackage[margin=1.5in]{geometry}

\usepackage[english]{babel}
\usepackage{comment}
\usepackage{color}

\let\phi\varphi

\let\epsilon\varepsilon
\let\subset\subseteq
\newcommand{\be}{\begin{equation*}}
 \newcommand{\ee}{\end{equation*}}
\newcommand{\bpf}{\begin{dimo}}
\newcommand{\epf}{\end{dimo}}
\newcommand{\bdefi}{\begin{defin}}
\newcommand{\edefi}{\end{defin}}
\newcommand{\bthm}{\begin{thm}}
\newcommand{\ethm}{\end{thm}}
\newcommand{\blem}{\begin{lem}}
\newcommand{\elem}{\end{lem}}
\newcommand{\bcor}{\begin{cor}}
\newcommand{\ecor}{\end{cor}}

\newcommand{\bprop}{\begin{prop}}
\newcommand{\eprop}{\end{prop}}
\newcommand{\bese}{\begin{ese}}
\newcommand{\eese}{\end{ese}}
\newcommand{\brem}{\begin{rem}}
\newcommand{\erem}{\end{rem}}
\newcommand{\bpfc}{\begin{dimoclaim}}
\newcommand{\epfc}{\end{dimoclaim}}

\newcommand{\wt}[1]{\widetilde{#1}}		
			
\newenvironment{quot}
{
	\vspace{-0.2cm}
	\vspace{0.2cm}
}

\theoremstyle{definition}
\newtheorem{d1}{Definition}[section] 

\newenvironment{defin}
{
	\begin{quot}
		\begin{d1}
		}
		{\end{d1}
	\end{quot}

}

\theoremstyle{definition}
\newtheorem{r1}[d1]{Remark}

\newenvironment{rem}
{
	\begin{quot}
		\begin{r1}
		}
		{\end{r1}
	\end{quot}
}

\theoremstyle{definition}
\newtheorem{e1}[d1]{Exercise}

\theoremstyle{definition}
\newtheorem{ese1}[d1]{Example}

\newenvironment{ese}
{
	\begin{quot}
		\begin{ese1}
	}
	{	
		\end{ese1}
	\end{quot}
}

\theoremstyle{definition}

\theoremstyle{definition}
\newtheorem{f2}[d1]{Fact}

\theoremstyle{definition}

\theoremstyle{definition}

\theoremstyle{definition}
\newtheorem{t1}[d1]{Theorem}

\newenvironment{thm}
{
	\begin{quot}
		\begin{t1}}
		{\end{t1}
	\end{quot}
}

\theoremstyle{definition}
\newtheorem*{T1*}{Theorem}

\newenvironment{teor*}
{
	\begin{quot}
		\begin{T1*}}
		{\end{T1*}
	\end{quot}
}

\newenvironment{dimo}
{\begin{proof}[Proof]
	}
	{\end{proof}}

\newenvironment{dimoclaim}{\emph{Proof of Claim:}\;}{\hfill$\square$}

	\theoremstyle{definition}
	\newtheorem{l1}[d1]{Lemma}
	
	\newenvironment{lem}
	{
		\begin{quot}
			\begin{l1}}
			{\end{l1}
		\end{quot}
	}
	\theoremstyle{definition}
	\newtheorem{p1}[d1]{Proposition}
	
	\newenvironment{prop}
	{
		\begin{quot}
			\begin{p1}}
			{\end{p1}
		\end{quot}
	}
	
	\theoremstyle{definition}
	\newtheorem{c1}[d1]{Corollary}
	
	\newenvironment{cor}
	{
		\begin{quot}
			\begin{c1}}
			{\end{c1}
		\end{quot}
	}

		\renewenvironment{abstract}
	{\list{}{\rightmargin\leftmargin}%
		\item[\textbf{Abstract:}]\relax}
	{\endlist}

 \newtheorem*{Theorem*}{Theorem}
 \newtheorem*{Proposition*}{Proposition}
 \newtheorem*{Lemma*}{Lemma}

\newcommand{\HHH}{\mathcal{H}}
\newcommand{\D}{\mathbb{D}}
\newcommand{\ML}{\mathcal{ML}}
\newcommand{\supp}{\mathrm{supp}}

\newcommand{\GCb}{\mathcal{G}\mathcal{C}_B}
\newcommand{\MLb}{\mathcal{ML}_B}
\newcommand{\GC}{\mathcal{G}\mathcal{C}}

\renewcommand{\tilde}{\widetilde}

\newtheorem{mainthm}{Theorem}

\makeatletter
\newcommand{\addresseshere}{%
  \enddoc@text\let\enddoc@text\relax
}
\makeatother

\title[Injectivity and continuity of the intersection form in infinite type]{Injectivity and continuity of the intersection form in infinite type}
\author{Tommaso Cremaschi}
\address{Department of Mathematics, Trinity College Dublin}
\email{cremasct@tcd.ie}
\author{D\'idac Mart\'inez-Granado}
\address{Department of Mathematics, University of Luxembourg}
\email{didac.martinezgranado@uni.lu}

\date{}
\thanks{}

\begin{document}
\maketitle
\begin{abstract}
We study two
fundamental properties of the intersection pairing between geodesic currents on infinite-type surfaces which are classical
in the finite-type closed setting: injectivity of the marked length spectrum
and continuity of the intersection form. We show that geodesic currents on infinite-type surfaces of the first kind are
determined by their intersection numbers with closed curves, extending
Otal's marked length spectrum rigidity argument from finite-type to infinite-type surfaces. In
contrast, we prove that the intersection pairing is not continuous in
general for the natural topologies by
constructing explicit counterexamples. We then prove sufficient criteria for continuity of intersection numbers for infinite-type surfaces.
\end{abstract}

\section{Introduction}

The geometry and dynamics of hyperbolic surfaces of infinite type have recently attracted
considerable attention; see for instance
\cite{AV20:Big,Bestvina2023:Nielsen}. In particular, the theory of infinite-type surfaces and their big
mapping class groups has seen substantial development in recent years (see e.g.~\cite{AFP23,MR23}).

A theme in this line of research is to understand which features
of the classical theory for finite-type surfaces persist in the
infinite-type setting, and which phenomena break down.
In this paper we focus on certain features
of Bonahon's theory of geodesic currents (see \cite{BS21:Boundary,Sar21:Traintracks, Sar22:Heights, BS24:Transverse,Sar24:QuadraticFunctioTheory} among many other references). First, we explore to what extent two structural results concerning the geometric intersection form for geodesic currents on closed surfaces (injectivity of its marked length spectrum and continuity) extend to 
infinite-type surfaces; second, we highlight limitations in the continuity that arise in this setting.

\medskip
\noindent
\textbf{Injectivity of the marked length spectrum.}
For a closed surface, a geodesic current is determined by its
intersection numbers with closed curves. In other words, the map sending
a current to its marked length spectrum is injective~\cite[Th\'eor\`eme~2]{Otal90:SpectreMarqueNegative}. We prove that the
same holds for infinite-type first-kind surfaces (i.e., surfaces equal to their convex core) by adapting Otal's
original argument to this setting, where some tools available in the compact surface setting are no longer present.

\begin{mainthm}[Injectivity of intersection numbers]
Let $S$ be an infinite-type surface of the first kind.
Let $\mathcal C$ denote the collection of all closed curves on $S$, and $\GC(S)$ the space of geodesic currents.
Then, the map
\[
\iota \colon \GC(S) \longrightarrow \mathbb{R}^{\mathcal C}, \qquad
\alpha \longmapsto (\iota(\alpha,\delta))_{\delta\in\mathcal C}
\]
is injective.
\label{mainthm:injective}
\end{mainthm}

The proof uses work of~Basmajian, Hakobyan and \v{S}ari\'c~\cite{BHS22:TypeProblem}
giving a criterion to detect ergodicity of the geodesic flow for surfaces of infinite type.
In particular, this gives examples of structures with dense orbits of the geodesic flow on infinite-type surfaces of the
first kind.

\medskip

\noindent
\textbf{Continuity of intersection number.}
In the finite-type and compact case, Bonahon proved that the geometric intersection
pairing extends continuously to geodesic currents with respect to the
weak$^*$-topology. For noncompact surfaces of finite type, continuity of the intersection number can already fail~\cite[Section~6]{Sas22:Currents}. We show that passing to infinite type introduces the same loss of compactness without cusps: all relevant intersections may move out through the ends of the surface. We give two main positive sufficient criteria for continuity. In Section~\ref{sec:positive_criteria} we give their proofs and some concrete variations that might be useful in more specific settings. 

\begin{mainthm}[Continuity under domination]
\label{mainthm:cont_domination}
Suppose that $\mu_n\to\mu$ and $\nu_n\to\nu$ weakly$^*$.  If there are
geodesic currents $\overline\mu,\overline\nu$ and constants $A,B>0$ such that
\[
 \mu_n\leq A\overline\mu,
 \qquad
 \nu_n\leq B\overline\nu
 \qquad\text{for every }n,
\]
and $\iota(\overline\mu,\overline\nu)<\infty$, then
\[
 \iota(\mu_n,\nu_n)\longrightarrow\iota(\mu,\nu).
\]
\end{mainthm}

The proof uses the intersection measure $\mathbf I(\mu,\nu) \coloneqq \mu \times \nu$ on the space of transverse pairs of geodesics on $S$. Its total mass is $i(\mu, \nu)$.
Weak$^*$-convergence of the two geodesic currents implies convergence of the $\mathbf I$ measures. Precisely, when all the intersection numbers are finite, $i(\mu_n, \nu_n)$ converges if and only if the measures $\mathbf I(\mu_n,\nu_n)$ are \emph{uniformly tight}, i.e., there is no escape of mass (see Section~\ref{subsec:tightness} for details).  

We emphasize that there also exist sequences of bounded measured laminations
$\alpha_n,\beta_n\to0$ in the uniform weak$^*$ topology for which
$\iota(\alpha_n,\beta_n)=1$ for every $n$ (Section~\ref{subsec:cont_fails}).

Theorem~\ref{mainthm:cont_domination} also allows us to extract another positive criterion stated directly in terms of properties of the two sequences. Fix a Borel fundamental section $F \subset \wt{X}$ for the action of $\pi_1(S)$. The \emph{$F$-incidence tail} of a current $\eta$ outside a compact $K \subset X$ counts the lifts which meet $F$ but whose projected geodesic is not contained in $K$. If these tails tend to zero uniformly for each of the two weak$^*$ convergent sequences, then their intersection numbers converge.
The condition of uniform decay can be interpreted as preventing mass
from escaping to infinity.

\begin{mainthm}[Positive criterion of continuity via uniform decay]
\label{mainthm:cont_uniformdecay}
Let $X$ be any hyperbolic surface.  Suppose that
$\mu_n\to\mu$ and $\nu_n\to\nu$ weakly$^*$.  If both sequences
$\{\mu_n\}$ and $\{\nu_n\}$ are uniformly $F$-decaying, then
\[
\iota(\mu_n,\nu_n)\longrightarrow\iota(\mu,\nu).
\]
\end{mainthm}
Outside of these hypotheses, one can again find counterexamples; see Section~\ref{subsec:cont_fails} below.

\subsection*{Related results}

Our main results provide new instances of phenomena from the geometry of finite-type surfaces that extend to the infinite-type setting, and some that fail. We briefly mention several related results in the literature which will be invoked or assumed throughout the paper.

\label{subsec:failures}

\noindent
\begin{tabularx}{\textwidth}{
  @{}
  >{\raggedright\arraybackslash}p{0.19\textwidth}
  >{\raggedright\arraybackslash}X
  >{\raggedright\arraybackslash}X
  @{}
}
\toprule
\textbf{Object} & \textbf{Extensions to infinite-type} & \textbf{Limitations in infinite-type} \\
\midrule

Measured laminations
&
Teichm\"uller boundary~\cite{BS21:Boundary} and train-track
descriptions~\cite{Sar21:Traintracks}.
&
Train-track charts need not recover the weak$^*$
topology~\cite[Remark~5.8]{Sar21:Traintracks}.
\\
\addlinespace

Quadratic differentials
&
Associated measured laminations for finite-area
differentials~\cite{Sar24:QuadraticFunctioTheory,
BXS26:quadraticdifferentialsrandomwalks,ST25:IntersectionNumbers}.
&
Associated laminations need not be Thurston
bounded~\cite[Theorem~8.1]{Sar22:Heights}.
\\
\addlinespace

Density
&
Simple closed curves are dense in certain
settings~\cite{Sar24:QuadraticFunctioTheory}.
&
Density can fail for surfaces of the second
kind~\cite{ST25:IntersectionNumbers}.
\\
\addlinespace

Bonahon's embedding
&
Proper in the uniform weak$^*$
topology~\cite{BS21:Boundary}.
&
Many natural sequences fail to converge in this
topology~\cite[Proposition~5]{BS21:Boundary}.
\\
\addlinespace
\hline

Intersection number
&
Continuous under strong conditions Theorems~\ref{mainthm:cont_domination} and \ref{mainthm:cont_uniformdecay}.
&
Not continuous in general in either topology;
see Examples~\ref{ese1} and~\ref{ese2}.
\\
\bottomrule
\end{tabularx}

\medskip

\textbf{Acknowledgments.}
The authors would like to thank Dragomir {\v{S}}ari\'{c} for enlightening conversations, for suggesting the argument for Theorem~\ref{denseorbit} and for giving comments on a first draft of this paper.
The first author was partially supported by MSCA 101107744–DefHyp and the second author by the Luxembourg National
Research Fund AFR/Bilateral-ReSurface 22/17145118 and the MSCA 101154865-CurrGeo.
The authors did not use AI or LLM tools for any aspect of this research or the writing of this manuscript.

\section{Background}

Throughout the paper, $X=\mathbb{H}/\Gamma$ denotes an infinite-type hyperbolic surface, where $\Gamma<\mathrm{PSL}(2,\mathbb R)$ is a Fuchsian group. We say that $X$ is of the \emph{first kind} if $X=\mathcal C(X)$, where $\mathcal C(X)$ denotes the convex core of $X$. By~\cite{BS19:GeodesicallyComplete}, this is equivalent to requiring that $X$ contain neither funnels nor half-planes.

We freely use both the upper half-plane model $\mathbb H$ and the disk model $\mathbb D$ of hyperbolic space, depending on the context.

Let $S$ be a surface with $\chi(S)<0$. By uniformization, $S$ admits hyperbolic structures, and we denote by $\mathcal T(S)$ its Teichm\"uller space. If $S$ is of infinite type, then $\mathcal T(S)$ is disconnected with respect to the Teichm\"uller metric, and its connected components correspond to quasiconformal equivalence classes. Accordingly, for a hyperbolic surface $X\in\mathcal T(S)$, we write $\mathcal T(X)$ for the connected component containing $X$.

Given an exhaustion of an infinite-type surface $S$ by compact subsurfaces $S_n$, we denote by $X_n\subset X$ the corresponding geodesic subsurfaces with totally geodesic boundary.

We now briefly recall the definitions and topologies on geodesic currents that will be used throughout the paper.

\subsection{Geodesic currents and their topologies}
\label{sec:meastheory}

\subsubsection{Geodesic currents}
\label{subsect:GeodesicCurr}

We consider a conformally hyperbolic Riemann surface $X$ of hyperbolic type with universal cover $\widetilde X$.

\textbf{Space of geodesics.}
We follow~\cite{BS21:Boundary}. Let $\partial_\infty \widetilde X$ denote the circle at infinity of $\widetilde X$. Any biholomorphic identification $\widetilde X\cong\D$ extends to a homeomorphism
\[
\widetilde X\cup\partial_\infty\widetilde X
\cong
\D\cup\partial\D.
\]

Each complete oriented geodesic of $\D$ is determined by its two endpoints in $\partial\D$. More generally, let $G(\widetilde X)$ denote the space of complete oriented geodesics of $\widetilde X$ endowed with its Poincar\'e metric. Identifying geodesics with their endpoints at infinity yields a natural identification
\[
G(\widetilde X)
=
\partial_\infty \widetilde X
\times
\partial_\infty \widetilde X
-\Delta,
\]
where
\[
\Delta=
\{(x,x)\,;\,x\in\partial_\infty \widetilde X\}
\]
is the diagonal. In particular, $G(\widetilde X)$ is homeomorphic to an open annulus.

The fundamental group $\pi_1(X)$ acts biholomorphically on $\widetilde X$, preserving the Poincar\'e metric, and acts on $G(\widetilde X)$ by homeomorphisms.

For infinite-type surfaces of the first kind, the space of geodesics is independent of the chosen hyperbolic structure up to equivariant homeomorphism. The following is proved in~\cite[Theorem~1.2]{Sar21:Traintracks}.

\bthm
\label{thm:homeo_geodesics}
Let $X_1$ and $X_2$ be two infinite-type hyperbolic surfaces of the first kind, and let $h \colon X_1 \to X_2$ be a homeomorphism. Then $h$ lifts to an equivariant homeomorphism
\[
\widetilde h \colon G(\widetilde X_1)\to G(\widetilde X_2).
\]
\ethm

There is a natural equivariant homeomorphism
\[
T^1\widetilde X \cong G(\widetilde X)\times\mathbb R,
\]
where $T^1\widetilde X$ denotes the unit tangent bundle of $\widetilde X$; see~\cite{AL17:Structures}.

A \emph{geodesic current} on $X$ is a $\pi_1(X)$-invariant Radon measure on $G(\widetilde X)$. Throughout the paper, all geodesic currents will be assumed to be symmetric, namely invariant under the involution of $G(\widetilde X)$ reversing the orientation of geodesics. We denote the space of geodesic currents by $\GC(X)$.

A \emph{measured lamination} is a geodesic current whose support consists of pairwise non-transversely intersecting geodesics. We denote the space of measured laminations by $\ML(X)$.

Theorem~\ref{thm:homeo_geodesics} induces homeomorphisms between the spaces $\ML(X_1)$ and $\ML(X_2)$, as well as between $\GC(X_1)$ and $\GC(X_2)$, via pushforward of measures.

Geodesic currents can equivalently be described in terms of geodesic-flow invariant measures on the unit tangent bundle. More precisely, if $\mu\in\GC(X)$, then the product measure $\mu\otimes dt$ defines a geodesic-flow invariant, orientation-invariant and $\pi_1(X)$-invariant Radon measure on $T^1\widetilde X$, where $dt$ denotes Lebesgue measure along the geodesic-flow direction. The same argument as in the finite-type setting~\cite[Page~3]{ES22:GeodesicCount} gives the details.

\subsubsection{Two topologies on the space of geodesic currents}

We consider two topologies on the space of geodesic currents, following~\cite{BS21:Boundary}.

First, the \emph{weak$^*$ topology} (also called the \emph{vague topology}) is defined by the seminorms
\[
|\alpha|_\xi
=
\left|
\int_{G(\widetilde X)} \xi\, d\alpha
\right|,
\]
where $\xi \colon G(\widetilde X)\to\mathbb R$ ranges over all continuous compactly supported functions.

We now define the \emph{uniform weak$^*$ topology}. For a continuous compactly supported function $\xi \colon G(\widetilde X)\to\mathbb R$, define
\[
\|\alpha\|_\xi
\coloneqq
\sup_{\phi\in\HHH(\widetilde X)}
\left|
\int_{G(\widetilde X)}
\xi\circ\phi \, d\alpha
\right|,
\]
where $\HHH(\widetilde X)$ denotes the group of biholomorphic automorphisms of $\widetilde X$.

A geodesic current $\alpha\in\GC(X)$ is called \emph{bounded} if $\|\alpha\|_\xi<\infty$ for every continuous compactly supported function $\xi$. We denote by $\GCb(X)$ the space of bounded geodesic currents.

The seminorms $\|\cdot\|_\xi$ induce the \emph{uniform weak$^*$ topology} on $\GCb(X)$. We denote by $\MLb(X)$ the space of bounded measured laminations, namely
\[
\MLb(X)=\ML(X)\cap\GCb(X).
\]

A sequence $\{\alpha_n\}_{n\in\mathbb N}\subset\GCb(X)$ converges to $\alpha\in\GCb(X)$ in the uniform weak$^*$ topology if and only if
\[
\sup_{\phi\in\HHH(\widetilde X)}
\left|
\int_{G(\widetilde X)}
\xi\circ\phi\, d\alpha_n
-
\int_{G(\widetilde X)}
\xi\circ\phi\, d\alpha
\right|
\to 0
\]
for every continuous compactly supported function $\xi\colon G(\widetilde X)\to\mathbb R$.

In particular, the uniform weak$^*$ topology is stronger than the weak$^*$ topology.

By~\cite[Theorem~1.3]{Sar21:Traintracks}, we have the following parametrization theorem.

\bthm\label{Sarconv}
Let $X$ be an infinite-type surface of the first kind and let $\Theta$ be a pants train track. Then the space $\mathcal{ML}(X,\Theta)$ of measured laminations weakly carried by $\Theta$, equipped with the weak$^*$ topology, is homeomorphic to the space $W(\Theta,[0,\infty))$ of edge weight systems on $\Theta$ endowed with the topology of pointwise convergence.

Similarly, the space $\mathcal{ML}_B(X,\Theta)$ of bounded measured laminations equipped with the uniform weak$^*$ topology is homeomorphic to the space $W_B(\Theta,[0,\infty))$ of bounded edge weight systems endowed with the topology of uniform convergence.
\ethm

\section{Injectivity of intersection number}

In this section we prove
that geodesic currents on infinite-type surfaces of first kind are determined by their marked length spectrum,
generalizing Otal's injectivity theorem. 

Given a geodesic current $\mu\in\GC(X)$, its \emph{marked intersection
spectrum} is the function
\[
  \iota(\mu,\cdot)\colon\mathcal C(X)\longrightarrow\mathbb R,
  \qquad c\longmapsto\iota(\mu,\delta_c).
\]
We prove that this function determines $\mu$ when $X$ is of the first
kind.

The following argument was suggested by \v{S}ari\'c.

\bthm\label{denseorbit}
Assume that $X=\mathbb H/\Gamma$ is of the first kind.  There is an
oriented geodesic $a\in G(\widetilde X)$ whose $\Gamma$-orbit is dense
in $G(\widetilde X)$.
\ethm

\bpf
If the geodesic flow on $T^1X$ is ergodic, then almost every orbit is
dense.  The lift of any such orbit determines a geodesic in
$G(\widetilde X)$ with dense $\Gamma$-orbit.

Suppose, therefore, that the geodesic flow on $T^1X$ is not ergodic.
Since $X$ is of the first kind, its convex core is all of $X$.  By
\cite[Theorem~3.4 and Proposition~3.7]{BS19:GeodesicallyComplete}, a
topological pants decomposition can be straightened to a geodesic pants
decomposition $\mathcal P$ of $X$.

Choose an exhaustion $X_1\subset X_2\subset\cdots$ by connected
finite-area geodesic subsurfaces made from pants in $\mathcal P$, with
$X_k\subset\operatorname{int}(X_{k+1})$.  Write
$\alpha_1^k,\ldots,\alpha_{n(k)}^k$ for the boundary cuffs of $X_k$.
The parabolicity criterion
\cite[Proposition~8.2]{BHS22:TypeProblem} says that the resulting
surface is parabolic if
\begin{equation}\label{eq:parabolic}
 \sum_{k=1}^{\infty}
 \frac{1}{\sum_{j=1}^{n(k)}
 \ell(\alpha_j^k)e^{\ell(\alpha_j^k)/2}}=\infty.
\end{equation}

Construct a first-kind hyperbolic surface $Y$, homeomorphic to $X$, by
giving every boundary cuff $\alpha_j^k$ length $1/n(k)$ and choosing
arbitrary twist parameters.  Then
\[
 \sum_{j=1}^{n(k)}\ell_Y(\alpha_j^k)
 e^{\ell_Y(\alpha_j^k)/2}
 =e^{1/(2n(k))}\leq e^{1/2}.
\]
Thus $Y$ satisfies~\eqref{eq:parabolic}; its geodesic flow is ergodic
and has a dense orbit.  A lift of this orbit gives a geodesic in
$G(\widetilde Y)$ with dense $\Gamma$-orbit.  Since $X$ and $Y$ are
homeomorphic surfaces of the first kind,
Theorem~\ref{thm:homeo_geodesics} supplies an equivariant homeomorphism
$G(\widetilde X)\cong G(\widetilde Y)$, which transfers the dense orbit
to $G(\widetilde X)$.
\epf

The intersection pairing is defined exactly as in
\cite[Proposition~4.5]{Bon86:EndsHyperbolicManifolds}; see also
\cite[Section~3.4.5]{AL17:Structures}. Since $X$ is locally compact, so
is $\mathbb PTX$, and therefore also the Whitney sum bundle
\[
\mathbb PTX\oplus\mathbb PTX.
\]
The subset
\begin{equation}
D(X)
=
\{
(x,\vec v,\vec w)\in\mathbb PTX\oplus\mathbb PTX
:
\vec v\neq \vec w
\}
\label{eq:DX}
\end{equation}
is open, hence locally compact in the subspace topology.

As in the finite-type setting, the intersection pairing need not be
finite for arbitrary geodesic currents on infinite-type surfaces.\footnote{For
instance, one may take two transverse pants decompositions endowed with
their counting measures.} Nevertheless, for geodesic currents induced by closed
geodesics one has
\[
\iota(\delta_\alpha,\delta_\beta)=i(\alpha,\beta),
\]
exactly as in Bonahon's construction.

Now we construct a new dense orbit $\gamma$ adapted to any given pair of geodesic currents $\mu, \nu$, in the sense that for any $\xi \in \partial_\infty\tilde X$ with $\xi=\gamma^+$ or $\xi=\gamma^-$, and any \emph{pencil of geodesics} $P(\xi) \coloneqq \{ \eta \in G(\wt{X}) : \eta^+ = \xi \text{ or } \eta^- = \xi \}$, we have $\mu(P(\xi))=0$ and $\nu(P(\xi))=0$.

\blem\label{lem:pencil_zero}
Suppose the $\Gamma$-action on $G(\mathbb H)$ has a dense orbit. Then, for any two arbitrary geodesic currents $\mu, \nu$, there is a geodesic $\gamma$ with dense orbit so that both measures vanish on the pencils based at either endpoint of every translate of $\gamma$.
\elem
\bpf
Let $\rho = \mu + \nu$.
Define
\[
E = \{ \xi \in \partial \mathbb H : \rho(P(\xi)) > 0\}~.
\]
We claim $E$ is countable. Indeed, we exhaust $G(\mathbb H)$ by compact sets $K_m$. For fixed $m,k$ only finitely many disjoint sets $P(\xi) \cap K_m$ can have measure at least $1/k$, since $\rho(K_m) < \infty$.
Taking the union over $m,k$ proves countability.
Let now $\{U_j\}$ be a countable basis for $G(\mathbb H)$. The set of dense-orbit points is
\[
\mathcal T = \cap_j \cup_{g \in \Gamma} g U_j.
\]
Each union is open and dense: it contains the given dense orbit $\gamma$. Thus $\mathcal T$ is a dense $G_\delta$ set.
The geodesics with either endpoint in the countable set $E$ form a countable union of closed nowhere-dense subsets. Baire's theorem, therefore supplies $a \in \mathcal T$, with both endpoints outside of $E$.
Finally, invariance gives $\rho(P(g \xi))=\rho(P(\xi))=0$.
\epf
The condition on the endpoint is necessary:
the current of a closed geodesic has an atom on the pencil based at
either endpoint of a lift.

We use the following measure-theoretic form of Otal's reconstruction
argument.

\bprop\label{prop:otal_reconstruction}
Suppose that $\gamma\in G(\mathbb H)$ has dense $\Gamma$-orbit.  Let $\mu$
and $\nu$ be $\Gamma$-invariant Radon measures on $G(\mathbb H)$ which
give zero mass to the pencils based at the endpoints of every translate
of $a$.  If
\[
  \iota(\mu,\delta_c)=\iota(\nu,\delta_c)
\]
for every closed geodesic $c$ on $\mathbb H/\Gamma$, then $\mu=\nu$.
\eprop

\bpf
Otal's four-axis construction
\cite[proof of Th\'eor\`eme~2]{Otal90:SpectreMarqueNegative} recovers
the measure of a box whose four sides are endpoints of translates of
$a$.  More precisely, short transverse arcs near those four translates
are approximated by fundamental segments of axes of hyperbolic
elements.  An alternating sum of the corresponding four intersection
numbers is twice the measure of the box, up to the measures of four
shrinking boundary strips.  Those error terms tend to zero because the
four boundary pencils have zero mass.  

The endpoints of the translates of $\gamma$ are dense.  The resulting boxes
form a basis for the topology of $G(\mathbb H)$; on
relatively compact open sets they determine a Radon measure.  Equality
on all these boxes therefore implies $\mu=\nu$.
\epf

\bthm\label{currdet}
Let $X$ be an infinite-type hyperbolic surface of the first kind.  The
map
\[
  \GC(X)\longrightarrow\mathbb R^{\mathcal C(X)},
  \qquad
  \mu\longmapsto
  \bigl(\iota(\mu,\delta_c)\bigr)_{c\in\mathcal C(X)}
\]
is injective.
\ethm

\bpf
By Theorem~\ref{denseorbit}, there is
$\gamma\in G(\widetilde X)$ with dense $\pi_1(X)$-orbit. Given any two geodesic currents $\mu, \nu$, by
Lemma~\ref{lem:pencil_zero}, there exists a dense geodesic $\gamma'$ so that both $\mu, \nu$ assign $0$ measure to pencils based at endpoints of lifts of $\gamma'$. By Proposition~\ref{prop:otal_reconstruction}, it thus follows that $\mu=\nu$, as desired.
\epf

\section{Continuity and escape of intersection mass}
\label{sec:positive_criteria}

In this section we give several positive criteria for the continuity of intersection number between geodesic currents for infinite-type surfaces. Roughly, the two criteria require uniform tightness of intersection number, which can be realized by requiring a uniform decay condition on the tails of the sequences of geodesic currents under consideration.

Finally, we give examples illustrating that, outside of the hypotheses, continuity might fail.

\subsection{Continuity under domination}

We write $\sigma\leq\tau$ if
$\sigma(E)\leq\tau(E)$ for every Borel subset of
$G(\widetilde X)$.

\bthm\label{thm:dominated_continuity}
Let $X$ be any hyperbolic surface.  Suppose that
$\mu_n\to\mu$ and $\nu_n\to\nu$ weakly$^*$.  If, for some geodesic currents
$\overline\mu,\overline\nu$ and constants $A,B>0$,
\[
 \mu_n\leq A\overline\mu,
 \qquad
 \nu_n\leq B\overline\nu
 \qquad\text{for every }n,
\]
and $\iota(\overline\mu,\overline\nu)<\infty$, then
$\iota(\mu_n,\nu_n)\to\iota(\mu,\nu)$.
\ethm

The proof will follow from a more precise description of how
intersection mass can escape.  This also explains why the hypothesis
involves the mutual intersection of the two dominating geodesic currents.

\subsection{The intersection measure}

Let
\[
 \mathcal I(\widetilde X)
 =
 \{(g,h)\in G(\widetilde X)^2:g\pitchfork h\}
\]
be the open set of ordered pairs of transversely intersecting geodesics.
The diagonal action of $\Gamma$ on this space is proper, since the
intersection point of a compact family of transverse pairs remains in a
compact subset of $\widetilde X$.  We write
\[
 \mathcal I(X)=\mathcal I(\widetilde X)/\Gamma.
\]
Sending a transverse pair to its intersection point and two tangent
lines identifies this quotient with
\[
 \bigl(\mathbb PTX\mathbin{\times_X}\mathbb PTX\bigr)\setminus\Delta,
\]
where $\Delta$ is the diagonal of equal tangent lines.

If $\mu,\nu\in\GC(X)$, the restriction of the product measure
$\mu\otimes\nu$ to $\mathcal I(\widetilde X)$ induces a Radon measure
$\mathbf I(\mu,\nu)$ on $\mathcal I(X)$.  We use the same normalization
as for the intersection pairing, so that
\begin{equation}\label{eq:intersection_total_mass}
 \mathbf I(\mu,\nu)(\mathcal I(X))=\iota(\mu,\nu).
\end{equation}
For two closed curves, the pushforward of this measure by the
intersection-point map is the sum of Dirac masses at their intersection
points, counted with multiplicity.

Although total mass is not continuous in general, the intersection
measure itself always varies weak-$^*$ continuously.

\bprop\label{prop:local_intersection_continuity}
Let $X$ be any hyperbolic surface.  If
$\mu_n\to\mu$ and $\nu_n\to\nu$ in the weak$^*$ topology on
$\GC(X)$, then
\[
 \mathbf I(\mu_n,\nu_n)\longrightarrow\mathbf I(\mu,\nu)
\]
weak-$^*$ as Radon measures on $\mathcal I(X)$.
\eprop

\bpf
First, weak-$*$ convergence of the two factors implies
\[
 \mu_n\otimes\nu_n\longrightarrow\mu\otimes\nu
\]
weak-$^*$ on $G(\widetilde X)^2$.  Indeed, the assertion is immediate for
compactly supported functions of the form $f(g)k(h)$, and finite sums of
such functions are uniformly dense on compact subsets.  Weak-$^*$
convergence also gives the local bounds on the masses needed to pass
from these finite sums to an arbitrary compactly supported continuous
function.

Let $q:\mathcal I(\widetilde X)\to\mathcal I(X)$ be the quotient map
and take $\varphi\in C_c(\mathcal I(X))$.  Properness of the action gives
a function $c\in C_c(\mathcal I(\widetilde X))$ such that
\[
 \sum_{\gamma\in\Gamma}c(\gamma z)=1
\]
for every $z\in q^{-1}(\supp\varphi)$.  Consequently,
\[
 \int_{\mathcal I(X)}\varphi\,d\mathbf I(\mu_n,\nu_n)
 =
 \int_{\mathcal I(\widetilde X)}
 c(z)\varphi(q(z))\,d(\mu_n\otimes\nu_n)(z).
\]
The integrand on the right has compact support contained in the open
set $\mathcal I(\widetilde X)$, and hence its extension by zero belongs
to $C_c(G(\widetilde X)^2)$.  The product convergence proves the
claim.
\epf

\subsection{A tightness criterion}
\label{subsec:tightness}
Suppose that $\mu_n\to\mu$ and $\nu_n\to\nu$, and set
$\mu_\infty=\mu$ and $\nu_\infty=\nu$.  We say that the pairs
$(\mu_n,\nu_n)$ are \emph{uniformly intersection-tight} if for every
$\epsilon>0$ there is a compact set $C\subset\mathcal I(X)$ such that
\begin{equation}\label{eq:intersection_tightness}
 \sup_{n\in\mathbb N\cup\{\infty\}}
 \mathbf I(\mu_n,\nu_n)(\mathcal I(X)\setminus C)<\epsilon.
\end{equation}

This condition has a direct geometric interpretation.  Let
$x(g,h)\in X$ be the intersection point and let
$\angle(g,h)\in(0,\pi/2]$ be the unoriented angle.  A family of finite
measures on $\mathcal I(X)$ is tight if and only if for every
$\epsilon>0$ there are a compact set $K\subset X$ and $\theta>0$ such
that
\begin{equation}\label{eq:geometric_tightness}
 \sup_{n\in\mathbb N\cup\{\infty\}}
 \mathbf I(\mu_n,\nu_n)
 \bigl(\{[g,h]:x(g,h)\notin K\text{ or }
 \angle(g,h)<\theta\}\bigr)<\epsilon.
\end{equation}
Thus one must rule out both escape through the ends of $X$ and escape
toward the non-transverse diagonal.

\bthm\label{thm:tight_continuity}
Let $X$ be any hyperbolic surface.  Suppose that
$\mu_n\to\mu$ and $\nu_n\to\nu$ weakly$^*$ and that
\[
 \iota(\mu_n,\nu_n)<\infty
 \quad\text{and}\quad
 \iota(\mu,\nu)<\infty.
\]
Then
\[
 \iota(\mu_n,\nu_n)\longrightarrow\iota(\mu,\nu)
\]
if and only if the pairs $(\mu_n,\nu_n)$ are uniformly
intersection-tight.
\ethm

\bpf
Write $\rho_n=\mathbf I(\mu_n,\nu_n)$ and
$\rho=\mathbf I(\mu,\nu)$.  Proposition
\ref{prop:local_intersection_continuity} gives $\rho_n\to\rho$ weak-$^*$.
Assume first that~\eqref{eq:intersection_tightness} holds.  Choose
$\chi\in C_c(\mathcal I(X))$ with $0\leq\chi\leq1$ and $\chi=1$ on
the compact set $C$ appearing there.  Then
\[
 \begin{split}
 |\rho_n(\mathcal I(X))-\rho(\mathcal I(X))|
 \leq{}&
 \left|\int\chi\,d\rho_n-\int\chi\,d\rho\right|\\
 &+\rho_n(\mathcal I(X)\setminus C)
 +\rho(\mathcal I(X)\setminus C).
 \end{split}
\]
The first term tends to zero and the last two are arbitrarily small.
Equation~\eqref{eq:intersection_total_mass} gives the desired
convergence.

Conversely, weak-$^*$ convergence together with convergence of the total
masses implies weak convergence of the finite Radon measures
$\rho_n$ to $\rho$.  Such a convergent sequence is uniformly tight.
Adding the single measure $\rho$ gives
\eqref{eq:intersection_tightness}.
\epf

Since uniform weak$^*$ convergence of geodesic currents implies weak$^*$
convergence, Theorem~\ref{thm:tight_continuity} applies without change
to the uniform weak$^*$ topology.

\begin{proof}[Proof of Theorem~\ref{thm:dominated_continuity}]
The weak$^*$ limits satisfy the same inequalities, and monotonicity of
product measures gives
\[
 \mathbf I(\mu_n,\nu_n)
 \leq AB\,\mathbf I(\overline\mu,\overline\nu).
\]
The measure on the right is a finite Radon measure and is therefore
tight.  Theorem~\ref{thm:tight_continuity} applies.
\end{proof}

\subsection{A uniform incidence-decay criterion}

We now formulate a sufficient condition directly on the two convergent
sequences.  Some care is needed in choosing the fundamental domain.  Fix
a locally finite triangulation of $X$ with compact simplex closures.
For every open simplex $e$, choose one lift $\widetilde e$, and set
\[
 F=\coprod_e\widetilde e\subset\widetilde X.
\]
Thus $F$ is a Borel fundamental section for the action of $\Gamma$ on
points: the restriction $F\to X$ is a Borel bijection.  Moreover, $F$ is
$\sigma$-compact and
\begin{equation}\label{eq:proper_fundamental_section}
 \overline{F\cap p^{-1}(K)}\quad\text{is compact in }\widetilde X
\end{equation}
for every compact set $K\subset X$, where
$p:\widetilde X\to X$ is the covering map.

Let
\[
 E_F=\{g\in G(\widetilde X):g\cap F\neq\varnothing\}
\]
and, for compact $K\subset X$, let
\[
 G(K)=\{g\in G(\widetilde X):p(g)\subset K\}.
\]
Both are Borel sets.  Indeed, $F$ is a countable union of compact sets,
and the set of geodesics meeting a fixed compact subset of
$\widetilde X$ is compact in $G(\widetilde X)$; the set $G(K)$ is
closed.  For a current $\eta$, define its \emph{$F$-incidence mass} and
its \emph{$F$-incidence tail outside $K$} by
\begin{equation}\label{eq:incidence_mass_tail}
 M_F(\eta)=\eta(E_F),
 \qquad
 \tau_K^F(\eta)=\eta\bigl(E_F\cap G(K)^c\bigr).
\end{equation}
We keep the section $F$ fixed.

Given a compact exhaustion $K_1\subset K_2\subset\cdots$ of $X$, where
$K_j\subset\operatorname{int}(K_{j+1})$ and every compact subset of $X$
is contained in some $K_j$, a sequence $\{\eta_n\}$ is \emph{uniformly
$F$-decaying} if
\begin{equation}\label{eq:uniform_F_decay}
 \lim_{j\to\infty}\sup_n\tau_{K_j}^F(\eta_n)=0.
\end{equation}
This condition is independent of the chosen compact exhaustion.  It is
substantially stronger than requiring only the coefficients of curves
outside $K_j$ to tend to zero: it also records how many lifts of those
curves meet the fixed fundamental section.

\bthm\label{thm:incidence_decay_continuity}
Let $X$ be any hyperbolic surface.  Suppose that
$\mu_n\to\mu$ and $\nu_n\to\nu$ weakly$^*$.  If both sequences
$\{\mu_n\}$ and $\{\nu_n\}$ are uniformly $F$-decaying, then
\[
 \iota(\mu_n,\nu_n)\longrightarrow\iota(\mu,\nu).
\]
\ethm

\brem
Uniform weak$^*$ convergence implies weak$^*$ convergence, so
Theorem~\ref{thm:incidence_decay_continuity} applies verbatim when both
sequences converge in the uniform weak$^*$ topology.  The uniform
$F$-decay assumption is not automatic, however: Examples~\ref{ese1}
and~\ref{ese2} converge in the uniform weak$^*$ topology but their
intersection numbers do not converge to the intersection numbers of
their limits.
\erem

\bpf[Proof of Theorem \ref{thm:incidence_decay_continuity}]
For a transverse pair $(g,h)$, let $x(g,h)\in\widetilde X$ be its
intersection point, and put
\[
 D_F=\{(g,h)\in\mathcal I(\widetilde X):x(g,h)\in F\}.
\]
This is a Borel fundamental domain for the diagonal $\Gamma$-action on
$\mathcal I(\widetilde X)$.  Hence
\begin{equation}\label{eq:incidence_intersection_bound}
 \iota(\eta,\zeta)=(\eta\otimes\zeta)(D_F)
 \leq M_F(\eta)M_F(\zeta).
\end{equation}

We first note that the incidence masses of both sequences are uniformly
bounded.  Choose $K=K_j$ so that the two tails in
\eqref{eq:uniform_F_decay} are at most $1$.  By
\eqref{eq:proper_fundamental_section}, the set
\[
 Q_K=\{g:g\cap\overline{F\cap p^{-1}(K)}\neq\varnothing\}
\]
is compact in $G(\widetilde X)$, and
$E_F\cap G(K)\subset Q_K$.  Weak$^*$ convergence uniformly bounds
$\mu_n(Q_K)$ and $\nu_n(Q_K)$.  Adding the tails gives a constant
$M<\infty$ such that
\begin{equation}\label{eq:uniform_incidence_mass_bound}
 \sup_n\max\{M_F(\mu_n),M_F(\nu_n)\}\leq M.
\end{equation}

Fix $\delta>0$ and choose $K=K_j$ so large that (the mass of) both tails is at most
$\delta$.  Decompose them into \emph{core} and \emph{tail} geodesic currents as
\[
 \mu_n=\mu_n^K+\mu_n^{\infty},
 \qquad
 \mu_n^K=\mu_n|_{G(K)},
\]
and define $\nu_n^K$ and $\nu_n^{\infty}$ similarly.  These are
geodesic currents, because $G(K)$ is $\Gamma$-invariant.  The core geodesic currents are
carried by $K$, while
\[
 M_F(\mu_n^{\infty})=\tau_K^F(\mu_n),
 \qquad
 M_F(\nu_n^{\infty})=\tau_K^F(\nu_n).
\]
Consequently, bilinearity and
\eqref{eq:incidence_intersection_bound} give
\begin{equation}\label{eq:core_tail_intersection_error}
 0\leq \iota(\mu_n,\nu_n)-\iota(\mu_n^K,\nu_n^K)
 \leq 2M\delta.
\end{equation}

The two families of core geodesic currents are relatively compact in the
weak$^*$ topology.  To see this, let $Q\subset G(\widetilde X)$ be
compact and choose a compact set $P\subset\widetilde X$ meeting every
geodesic in $Q$, for instance the set of closest points to a fixed
basepoint.  The lifted triangulation is locally finite, so $P$ meets
only finitely many translates of the chosen lifted simplices.  It
follows that, for some $\gamma_1,\ldots,\gamma_{N_Q}\in\Gamma$,
\[
 Q\subset\bigcup_{r=1}^{N_Q}\gamma_rE_F.
\]
Every current $\eta$ is $\Gamma$-invariant, and hence
\[
 \eta(Q)\leq N_QM_F(\eta).
\]
The bound~\eqref{eq:uniform_incidence_mass_bound} now gives uniform
bounds on every compact subset of $G(\widetilde X)$.  Since
$G(\widetilde X)$ is locally compact and second countable, the standard
weak-$^*$ compactness theorem proves the claimed relative compactness;
$\Gamma$-invariance passes to the limits.  Since $G(K)$ is closed, every
limit of core geodesic currents is still carried by $K$.

Pass to a subsequence which realizes the limit superior of
$\iota(\mu_n,\nu_n)$, and then to a further subsequence such that
\[
 \mu_n^K\longrightarrow\widehat\mu,
 \qquad
 \nu_n^K\longrightarrow\widehat\nu.
\]
Both limits are carried by $K$.  Continuity of intersection when one
factor has \emph{compact carrier}, i.e. the projection of its support on $G(\tilde{X})$ is contained in a compact subset of $X$ (\cite[Theorem~2.4]{BIPP21:CurrentsSystoles}) gives
\begin{equation}\label{eq:compact_core_intersection_limit}
 \iota(\mu_n^K,\nu_n^K)\longrightarrow
 \iota(\widehat\mu,\widehat\nu).
\end{equation}
On the other hand, $\mu_n^K\leq\mu_n$ and $\nu_n^K\leq\nu_n$.
Passing to the limit against nonnegative compactly supported functions
shows that
\[
 \widehat\mu\leq\mu,
 \qquad
 \widehat\nu\leq\nu.
\]
Monotonicity of the intersection measure therefore implies
\[
 \iota(\widehat\mu,\widehat\nu)\leq\iota(\mu,\nu).
\]
Together with~\eqref{eq:core_tail_intersection_error} and
\eqref{eq:compact_core_intersection_limit}, this proves
\[
 \limsup_{n\to\infty}\iota(\mu_n,\nu_n)
 \leq\iota(\mu,\nu)+2M\delta.
\]
Letting $\delta\to0$ gives the required upper bound.  For the reverse
bound, Proposition~\ref{prop:local_intersection_continuity} and lower
semicontinuity of total mass under weak-$^*$-convergence give
\[
 \iota(\mu,\nu)
 \leq\liminf_{n\to\infty}\iota(\mu_n,\nu_n).
\]
The two inequalities finish the proof.
\epf

For weighted curves, the domination hypothesis has the following
concrete form.

\bcor\label{cor:summable_envelopes}
Let $\{\alpha_i\}$ and $\{\beta_j\}$ be countable families of closed
curves, and let $A_i,B_j\geq0$ be such that
\[
 \overline\mu=\sum_i A_i\delta_{\alpha_i},
 \qquad
 \overline\nu=\sum_j B_j\delta_{\beta_j}
\]
define geodesic currents.  Suppose moreover that
\begin{equation}\label{eq:summable_intersection_envelope}
 \sum_{i,j}A_iB_j i(\alpha_i,\beta_j)<\infty.
\end{equation}
If $0\leq a_i^{(n)}\leq A_i$ and
$0\leq b_j^{(n)}\leq B_j$ converge pointwise to $a_i$ and $b_j$, then
\[
 \mu_n=\sum_i a_i^{(n)}\delta_{\alpha_i}
 \quad\text{and}\quad
 \nu_n=\sum_j b_j^{(n)}\delta_{\beta_j}
\]
converge weakly$^*$ to the corresponding limit geodesic currents, and
\[
 \iota(\mu_n,\nu_n)\longrightarrow\iota(\mu,\nu).
\]
\ecor

\bpf
Dominated convergence on every compact subset of
$G(\widetilde X)$ gives weak$^*$ convergence.  By Tonelli's theorem,
the left-hand side of~\eqref{eq:summable_intersection_envelope} equals
$\iota(\overline\mu,\overline\nu)$.  Theorem
\ref{thm:dominated_continuity} concludes the proof.
\epf

Condition~\eqref{eq:summable_intersection_envelope}, rather than
separate summability of the two coefficient sequences, is the natural
decay condition: it incorporates the geometry of the intersections at
infinity.

\subsection{Changing the hyperbolic structure}
The preceding results vary the two geodesic currents. Now we consider the case when one of the geodesic currents is a Liouville current, and the other is a fixed measured lamination. For
$X\in\mathcal T(Z)$, let $L_X$ denote the Liouville current, pulled back
using the marking. In this case, we have
\[
 \ell_X(\beta)=\iota(L_X,\beta)
\]
whenever this intersection is finite.

\bprop\label{prop:fixed_lamination_length}
Let $Z$ be a hyperbolic structure on $S$, and let
$\beta\in\ML(Z)$ satisfy $\ell_Z(\beta)<\infty$.  Then
$\ell_X(\beta)<\infty$ for every $X\in\mathcal T(Z)$.

Moreover, for all
$X,Y\in\mathcal T(Z)$,
\begin{equation}\label{eq:fixed_lamination_length_comparison}
 e^{-2d_{\mathrm T}(X,Y)}\ell_X(\beta)
 \leq
 \ell_Y(\beta)
 \leq
 e^{2d_{\mathrm T}(X,Y)}\ell_X(\beta).
\end{equation}

In particular, $X\mapsto\ell_X(\beta)$ is continuous on
$\mathcal T(Z)$.

If $\beta\neq0$, then
$X\mapsto\log\ell_X(\beta)$ is $2$-Lipschitz.

\eprop

\bpf
The comparison between the Thurston length metric and the
Teichm\"uller metric gives
\eqref{eq:fixed_lamination_length_comparison}.  See
\cite{Sor73:Dilatation}.  Taking $Y=Z$ first proves finiteness throughout
$\mathcal T(Z)$.  If $X_n\to X$, then
\[
 \bigl|\ell_{X_n}(\beta)-\ell_X(\beta)\bigr|
 \leq
 \bigl(e^{2d_{\mathrm T}(X_n,X)}-1\bigr)\ell_X(\beta)
 \longrightarrow0.
\]
The logarithmic Lipschitz estimate is another direct consequence of
\eqref{eq:fixed_lamination_length_comparison}.\qedhere
\epf

\brem
Proposition~\ref{prop:fixed_lamination_length} concerns continuity in
the Teichm\"uller variable with $\beta$ fixed.  Note that if one fixes the Teichm\"uller variable and varies the measured lamination weakly$^*$ continuously, the continuity statement fails: length
can escape through the ends in that topology.
\erem

\subsection{Failure of continuity}
\label{subsec:cont_fails}
\leavevmode

We now show that the intersection pairing is not continuous in general.
Both examples can be realized on a first-kind ladder surface of bounded
geometry.  Choose its hyperbolic structure and a pants train track
$\Theta$ periodically, so that the constants in the local train-track
charts are uniform over all handles.

Let $\alpha$ be a simple geodesic arc running from one genus end of
the ladder to the other.  In the $i$th handle choose a meridian $m_i$
so that the curves $m_i$ are pairwise disjoint and
\[
  i(\alpha,m_i)=1.
\]

The track $\Theta$ may be chosen to carry $\alpha$ and all $m_i$, with
disjoint local branches for the different meridians and uniformly
bounded weights for $\alpha$.

\bese\label{ese1}
Set
\[
  \beta_n=\frac1n\sum_{i=1}^n\delta_{m_i}.
\]
The largest branch weight of $\beta_n$ on $\Theta$ is $1/n$.
Theorem~\ref{Sarconv} therefore gives $\beta_n\to0$ in the uniform
weak$^*$ topology.  Nevertheless,
\[
  \iota(\delta_\alpha,\beta_n)
  =\frac1n\sum_{i=1}^n i(\alpha,m_i)=1.
\]
Thus intersection is not continuous even when one factor is fixed.

The intersection measure assigns mass $1/n$ to the intersection in each
of the first $n$ handles.  It therefore converges weak-$^*$ to zero but is
not tight.  Equivalently, the infinite sum
\[
 \overline\beta=\sum_{i=1}^{\infty}\frac1i\delta_{m_i}
\]
satisfies
$\iota(\delta_\alpha,\overline\beta)=\sum_i1/i=\infty$, so
Theorem~\ref{thm:dominated_continuity} does not apply.
\eese

For joint discontinuity, work in each periodic handle, and consider generators $a_i, b_i$ for each handle, with $i(a_i, b_i)=1$ for each $i \in \mathbb{N}$.

\bese\label{ese2}
Define
\[
  \mu_n=\frac1{\sqrt n}\,\sum_{i=1}^n\delta_{a_i},
  \qquad
  \nu_n=\frac1{\sqrt n}\,\sum_{i=1}^n\delta_{b_i},
\]
 By Theorem~\ref{Sarconv}, both sequences
converge to zero in the uniform weak$^*$ topology.  On the other hand,
\[
 \iota(\mu_n,\nu_n)
 =\frac1{n}\sum_{i=1}^n i(a_i,b_i)=1.
\]
Hence the pairing is not jointly continuous in the uniform weak$^*$
topology, and therefore not in the weak$^*$ topology either.
\eese

The next observation rules out a repair based
only on coefficients of curves outside a compact exhaustion.

\bprop\label{prop:coefficient_decay}
There are sequences of weighted simple closed curves $\lambda_n$ and
$\eta_n$ which converge weakly$^*$ to zero, whose coefficients tend to
zero, and for which $\iota(\lambda_n,\eta_n)=1$ for every $n$.
\eprop

\bpf
Choose pairwise disjoint one-holed tori $T_n$ which leave every compact
subset of $X$.  In $T_n$, choose primitive slope curves $a_n,b_n$ with
$i(a_n,b_n)=n^2$, and put
\[
  \lambda_n=\frac1n\delta_{a_n},
  \qquad
  \eta_n=\frac1n\delta_{b_n}.
\]
The supports escape every compact set, so both sequences converge
weakly$^*$ to zero.  Their coefficients are $1/n$, while
\[
  \iota(\lambda_n,\eta_n)
  =\frac1{n^2}i(a_n,b_n)=1.
\]
The intersection measure has mass one in $T_n$, making the failure of
tightness explicit.
\epf

\brem
Proposition~\ref{prop:coefficient_decay} shows that small tail
coefficients alone do not control intersection: local geometric
intersection can grow faster than the coefficients decay.  For the
infinite sums
\[
\overline\lambda=\sum_n\frac1n\delta_{a_n},
 \qquad
 \overline\eta=\sum_n\frac1n\delta_{b_n},
\]
one has
\[
 \iota(\overline\lambda,\overline\eta)
 \geq\sum_n\frac1{n^2}i(a_n,b_n)=\infty.
\]
Thus the additional finite-intersection hypothesis in
Theorem~\ref{thm:dominated_continuity} is indispensable.  Nor do these
sequences contradict Theorem~\ref{thm:incidence_decay_continuity}: once
$T_n$ lies outside $K$, inequality
\eqref{eq:incidence_intersection_bound} gives
\[
 1\leq\tau_K^F(\lambda_n)\tau_K^F(\eta_n),
\]
so the two incidence tails cannot both decay uniformly.
\erem

\thispagestyle{empty}
{\footnotesize
\markboth{References}{References}
\bibliographystyle{alpha}
\bibliography{mybib}{}

@article {BIPP21:CurrentsSystoles,
    AUTHOR = {Burger, Marc and Iozzi, Alessandra and Parreau, Anne and Pozzetti, Maria Beatrice},
     TITLE = {Currents, systoles, and compactifications of character
              varieties},
   JOURNAL = {Proc. Lond. Math. Soc. (3)},
  FJOURNAL = {Proceedings of the London Mathematical Society. Third Series},
    VOLUME = {123},
      YEAR = {2021},
    NUMBER = {6},
     PAGES = {565--596},
      ISSN = {0024-6115},
   MRCLASS = {32G15 (22E40 57K20)},
  MRNUMBER = {4368682},
       DOI = {10.1112/plms.12419},
       URL = {https://doi.org/10.1112/plms.12419},
}

@article{BHS22:TypeProblem,
author = {Basmajian, Ara and Hakobyan, Hrant and {\v{S}}ari\'{c}, Dragomir},
title = {The type problem for Riemann surfaces via {F}enchel--{N}ielsen parameters},
journal = {Proceedings of the London Mathematical Society},
volume = {125},
number = {3},
pages = {568--625},
doi = {10.1112/plms.12465},
url = {https://londmathsoc.onlinelibrary.wiley.com/doi/abs/10.1112/plms.12465},
eprint = {https://londmathsoc.onlinelibrary.wiley.com/doi/pdf/10.1112/plms.12465},
year = {2022}
}

@article {Sor73:Dilatation,
    AUTHOR = {Sorvali, Tuomas},
     TITLE = {On the dilatation of isomorphisms between covering groups},
   JOURNAL = {Ann. Acad. Sci. Fenn. Ser. A. I.},
  FJOURNAL = {Annales Academiae Scientiarum Fennicae. Series A. I.
              Mathematica},
      YEAR = {1973},
    VOLUME = {551},
     PAGES = {15},
      ISSN = {0066-1953},
   MRCLASS = {30A58},
  MRNUMBER = {342694},
MRREVIEWER = {C.\ Earle},
}

@incollection{AFP23,
  author    = {Aramayona, Javier and Fossas, Ariadna and Parlier, Hugo},
  title     = {Big mapping class groups: an overview},
  booktitle = {In the Tradition of Thurston IV},
  series    = {Geometry and Topology Monographs},
  volume    = {24},
  pages     = {63--108},
  publisher = {Mathematical Sciences Publishers},
  year      = {2023},
  doi       = {10.2140/gtm.2023.24.63}
}

@article{MR23,
  author  = {Mann, Kathryn and Rafi, Kasra},
  title   = {Large-scale geometry of big mapping class groups},
  journal = {Geometry \& Topology},
  volume  = {27},
  number  = {6},
  pages   = {2237--2296},
  year    = {2023},
  doi     = {10.2140/gt.2023.27.2237}
}

@book {ES22:GeodesicCount,
    AUTHOR = {Viveka Erlandsson and Juan Souto},
     TITLE = {Geodesic currents and {M}irzakhani's curve counting},
    SERIES = {Progress in Mathematics},
    VOLUME = {xii-151},
 PUBLISHER = {Birkhäuser Cham},
      YEAR = {2022},
     PAGES = {226},
      ISBN = {978-3-031-08704-2},
       DOI = {https://doi.org/10.1007/978-3-031-08705-9},
}

@misc{ST25:IntersectionNumbers,
      title={Intersection numbers between horizontal foliations of quadratic differentials}, 
      author={Dragomir {\v{S}}ari{\'c} and Taro Shima},
      year={2025},
      eprint={2506.14943},
      archivePrefix={arXiv},
      primaryClass={math.CV},
      url={https://arxiv.org/abs/2506.14943}, 
}

@misc{BXS26:quadraticdifferentialsrandomwalks,
  title         = {Quadratic differentials and random walks on the dual graph of a pants decomposition},
  author        = {Charles Bordenave and Xinlong Dong and Dragomir {\v{S}}ari{\'c}},
  year          = {2026},
  eprint        = {2604.11985},
  archivePrefix = {arXiv},
  primaryClass  = {math.DS},
  url           = {https://arxiv.org/abs/2604.11985},
}

@misc{Sar24:QuadraticFunctioTheory,
      title={Quadratic differentials and function theory on Riemann surfaces}, 
      author={Dragomir \v{S}ari{\'c}},
      year={2024},
      eprint={2407.16333},
      archivePrefix={arXiv},
      primaryClass={math.DS},
      url={https://arxiv.org/abs/2407.16333}, 
}

@incollection {AL17:Structures,
    AUTHOR = {Aramayona, Javier and Leininger, Christopher J.},
     TITLE = {Hyperbolic structures on surfaces and geodesic currents},
 BOOKTITLE = {Algorithmic and Geometric Topics around Free Groups and
              Automorphisms},
    SERIES = {Adv. Courses Math. CRM Barcelona},
     PAGES = {111--149},
 PUBLISHER = {Birkh\"{a}user/Springer, Cham},
      YEAR = 2017,
}

@article {BS21:Boundary,
    AUTHOR = {{B}onahon, Francis and {\v{S}}ari\'{c}, Dragomir},
     TITLE = {A {T}hurston boundary for infinite-dimensional
              {T}eichm\"{u}ller spaces},
   JOURNAL = {Math. Ann.},
  FJOURNAL = {Mathematische Annalen},
    VOLUME = {380},
      YEAR = {2021},
    NUMBER = {3-4},
     PAGES = {1119--1167},
      ISSN = {0025-5831,1432-1807},
   MRCLASS = {57K20 (30F60)},
  MRNUMBER = {4297183},
MRREVIEWER = {Felipe\ A.\ Franco},
       DOI = {10.1007/s00208-021-02148-z},
       URL = {https://doi.org/10.1007/s00208-021-02148-z},
}

@article{Bon86:EndsHyperbolicManifolds,
title={Bouts des Vari\'et\'es Hyperboliques de Dimension 3},
volume={124},
DOI={10.2307/1971388},
number={1},
journal={Ann. of Math. (2)},
fjournal={The Annals of Mathematics},
author={Bonahon, Francis},
year={1986},
pages={71--158}}

@incollection{AV20:Big,
  author    = {Aramayona, Javier and Vlamis, Nicholas G.},
  title     = {Big Mapping Class Groups: An Overview},
  booktitle = {In the Tradition of Thurston: Geometry and Topology},
  editor    = {Canary, Richard D. and Leininger, Christopher J. and Margalit, Dan and Schleimer, Saul},
  series    = {Springer Proceedings in Mathematics \& Statistics},
  volume    = {329},
  publisher = {Springer},
  year      = {2020},
  pages     = {459--496},
  chapter   = {12},
  doi       = {10.1007/978-3-030-45187-0_12}
}

@article {BS24:Transverse,
    AUTHOR = {Bestvina, Mladen and Rasmussen, Alexander J.},
     TITLE = {Transverse measures to infinite type laminations},
   JOURNAL = {Comment. Math. Helv.},
  FJOURNAL = {Commentarii Mathematici Helvetici. A Journal of the Swiss
              Mathematical Society},
    VOLUME = {99},
      YEAR = {2024},
    NUMBER = {4},
     PAGES = {717--767},
      ISSN = {0010-2571,1420-8946},
   MRCLASS = {57K20 (30F60 46A55)},
  MRNUMBER = {4815164},
       DOI = {10.4171/cmh/575},
       URL = {https://doi.org/10.4171/cmh/575},
}

@article {Sar22:Heights,
    AUTHOR = {\v{S}ari\'c, Dragomir},
     TITLE = {The heights theorem for infinite {R}iemann surfaces},
   JOURNAL = {Geom. Dedicata},
  FJOURNAL = {Geometriae Dedicata},
    VOLUME = {216},
      YEAR = {2022},
    NUMBER = {3},
     PAGES = {Paper No. 33, 23},
      ISSN = {0046-5755,1572-9168},
   MRCLASS = {30F30},
  MRNUMBER = {4411726},
MRREVIEWER = {Vasilisa\ Shramchenko},
       DOI = {10.1007/s10711-022-00694-7},
       URL = {https://doi.org/10.1007/s10711-022-00694-7},
}

@article{Sas22:Currents,
  author    = {Dounnu Sasaki},
  title     = {Currents on cusped hyperbolic surfaces and denseness property},
  journal   = {Groups, Geometry, and Dynamics},
  volume    = {16},
  number    = {3},
  pages     = {1077--1117},
  year      = {2022},
  doi       = {10.4171/GGD/688},
  url       = {https://ems.press/journals/ggd/articles/7869839}
}

@misc{Bestvina2023:Nielsen,
      title={Towards {N}ielsen-{T}hurston classification for surfaces of infinite type}, 
      author={Mladen Bestvina and Federica Fanoni and Jing Tao},
      year={2023},
      eprint={2303.12413},
      archivePrefix={arXiv},
      primaryClass={math.GT},
      url={https://arxiv.org/abs/2303.12413}, 
}

@article {Sar21:Traintracks,
    AUTHOR = {\v{S}ari\'{c}, Dragomir},
     TITLE = {Train tracks and measured laminations on infinite surfaces},
   JOURNAL = {Trans. Amer. Math. Soc.},
  FJOURNAL = {Transactions of the American Mathematical Society},
    VOLUME = {374},
      YEAR = {2021},
    NUMBER = {12},
     PAGES = {8903--8947},
      ISSN = {0002-9947},
   MRCLASS = {30F60 (30F45)},
  MRNUMBER = {4337933},
       DOI = {10.1090/tran/8488},
       URL = {https://doi.org/10.1090/tran/8488},
}

@Article{Otal90:SpectreMarqueNegative,
    AUTHOR = {Otal, Jean-Pierre},
     TITLE = {Le Spectre Marqu{\'e} des Longuers des Surfaces {\`a} Courbure N{\'e}gative },
   JOURNAL = {Ann. of Math. (2)},
  FJOURNAL = {The Annals of Mathematics},
    VOLUME = 131,
      YEAR = 1990,
    NUMBER = 1,
     PAGES = {151}
}

@article{BS19:GeodesicallyComplete, 
title={Geodesically Complete Hyperbolic Structures}, 
volume={166}, 
DOI={10.1017/S0305004117000792}, 
number={2}, 
journal={Mathematical Proceedings of the Cambridge Philosophical Society}, 
publisher={Cambridge University Press}, 
author={Basmajian, Ara and {\v{S}}ari\'{c}, Dragomir}, 
year={2019}, 
pages={219-242}}

}
	\end{document}